\documentclass[11pt, a4paper]{article} %{bxjsarticle}

\usepackage{amssymb}
\usepackage{amsthm}
\usepackage{amsmath}
\usepackage{enumerate}
\usepackage{mathtools}
\usepackage{bm}
\mathtoolsset{showonlyrefs=true} % if an equation is not refereed from anywhere, no equation number is shown

\usepackage{booktabs}
\newcommand{\ra}[1]{\renewcommand{\arraystretch}{#1}}
\usepackage{float}
\usepackage{graphicx}
\usepackage{adjustbox}

\usepackage{algorithm}

\usepackage{easy-todo}

\DeclarePairedDelimiter\ceil{\lceil}{\rceil}

\usepackage[
driverfallback=dvipdfmx,
bookmarks=true,
bookmarksnumbered=true,
bookmarkstype=toc,
setpagesize=false,
pdftitle={Novel DSPG},
pdfauthor={NAMCHAISIRI Charles},
pdfkeywords={semidefinite programs, spectral projected gradient method, dual problem, generalized problem}
]{hyperref}

\usepackage{fancybox,color}
\definecolor{lred}{rgb}{1,0.8,0.5}
\definecolor{lblue}{rgb}{0.8,0.8,1}
\definecolor{dred}{rgb}{0.6,0,0}
\definecolor{dblue}{rgb}{0,0,0.7}
\definecolor{violet}{rgb}{0.5804,0.0000,0.8275}
\definecolor{purple}{rgb}{0.2400,0.5700,0.2500}
\definecolor{TGreen}{rgb}{0,0.50,0.10}
\DeclareUnicodeCharacter{2212}{-}
\def\@themcountersep{}
\newtheorem{THEO}{Theorem}[]
\newtheorem{ASSM}[THEO]{Assumption}

\newtheorem{CORO}[THEO]{Corollary}

\newtheorem{LEMM}[THEO]{Lemma}

\def\0{\mbox{\bf 0}}
\def\1{\mbox{\bf 1}}
\def\2{\mbox{\bf 2}}
\def\3{\mbox{\bf 3}}
\def\4{\mbox{\bf 4}}
\def\5{\mbox{\bf 5}}
\def\6{\mbox{\bf 6}}
\def\7{\mbox{\bf 7}}
\def\8{\mbox{\bf 8}}
\def\9{\mbox{\bf 9}}
\def\a{\mbox{\boldmath $a$}}
\def\b{\mbox{\boldmath $b$}}

\newdimen\Vhige \Vhige=0pt
\def\chige#1{{\setbox\Vhige\hbox{#1}\ifdim\ht\Vhige=1ex\accent24 #1%
  \else\ooalign{\unhbox\Vhige\crcr\hidewidth\char24\hidewidth}\fi}}

\def\e{\mbox{\boldmath $e$}}
\def\f{\mbox{\boldmath $f$}}

\def\x{\mbox{\boldmath $x$}}
\def\y{\mbox{\boldmath $y$}}
\def\y{\mbox{\boldmath $y$}}
\def\z{\mbox{\boldmath $z$}}
\def\A{\mbox{\boldmath $A$}}

\def\C{\mbox{\boldmath $C$}}
\def\D{\mbox{\boldmath $D$}}

\def\I{\mbox{\boldmath $I$}}

\def\L{\mbox{\boldmath $L$}}

\def\O{\mbox{\boldmath $O$}}

\def\S{\mbox{\boldmath $S$}}

\def\U{\mbox{\boldmath $U$}}

\def\X{\mbox{\boldmath $X$}}

\def\AC{\mbox{$\cal A$}}
\def\BC{\mbox{$\cal B$}}

\def\DC{\mbox{$\cal D$}}

\def\FC{\mbox{$\cal F$}}

\def\MC{\mbox{$\cal M$}}
\def\NC{\mbox{$\cal N$}}

\def\PC{\mbox{$\cal P$}}
\def\QC{\mbox{$\cal Q$}}

\def\SC{\mbox{$\cal S$}}

\def\UC{\mbox{$\cal U$}}

\def\BSigma{\mbox{\boldmath $\Sigma$}}

\def\Real{\mbox{$\mathbb{R}$}}

\def\SMAT{\mbox{$\mathbb{S}$}}

\def\bS{\mathbb{S}}

\usepackage[normalem]{ulem}
\ifx11 % 11 => use these commands, 10 => NOT use these commands

\fi
\title{Dual Spectral Projected Gradient Method for Generalized Log-det Semidefinite Programming}

\author{Charles Namchaisiri\thanks{School of Computing, Tokyo Institute of Technology,  Japan. (\href{namchaisiri.c.aa@m.titech.ac.jp}{namchaisiri.c.aa@m.titech.ac.jp})}
\and Makoto Yamashita\thanks{School of Computing, Tokyo Institute of Technology,  Japan. (\href{Makoto.Yamashita@c.titech.ac.jp}{Makoto.Yamashita@c.titech.ac.jp})}  }
\numberwithin{equation}{section}
\allowdisplaybreaks

\begin{document}

\maketitle

\begin{abstract}
Log-det semidefinite programming (SDP) problems are optimization problems that often arise from Gaussian graphic models. A log-det SDP problem with an $\ell_1$-norm term has been examined in many methods, and  
the dual spectral projected gradient (DSPG) method by Nakagaki et al.~in 2020 is designed to efficiently solve the dual problem of the log-det SDP by combining a non-monotone line-search projected gradient method with the step adjustment for positive definiteness. In this paper, we extend the DSPG method for solving a generalized log-det SDP problem involving additional terms to cover more structures in Gaussian graphical models in a unified style. We establish the convergence of the proposed method to the optimal value. We conduct numerical experiments to illustrate the efficiency of the proposed method. 
\end{abstract}

\section{Introduction}\label{sec:introduction}
In this paper, we address the following log-determinant semidefinite programming (SDP) optimization problem:
\begin{equation}\label{primal}
\begin{split}
 \min_{\X\in\SMAT^n} & \ \ f(\X) := \C \bullet \X - \mu \log \det \X + \sum_{h=1}^{H} \lambda_h \|\QC_h(\X)\|_{p_h} \\
\mbox{s.t.}  & \ \  \AC(\X) = \b, \X \succ \O.
\end{split}
\tag{$\PC$}
\end{equation}
We use $\Real^n$ and $\SMAT^n$ to denote the sets of 
$n$-dimensional vectors and $n \times n$ symmetric matrices, respectively.
The inner product between 
$\C \in \SMAT^n$ and $\X \in \SMAT^n$ is defined by
$\C \bullet \X := \sum_{i=1}^n \sum_{j=1}^n C_{ij} X_{ij}$.
We use nonnegative $\mu$ and $\lambda_1, \dots, \lambda_H$ 
as weight parameters in the objective function,
${\QC}_{h}:\SMAT^n \to \Real^{n_h}$ for each $h = 1, \ldots, H$ is a linear map, $\|\y\|_{p_h}:= \left(\sum_{i=1}^{n_h} y_i^{p_h}\right)^{\frac{1}{p_h}}$ is the $\ell_{p_h}$-norm
of $\y \in \Real^{n_h}$ with $p_h \ge 1$.
In the constraints, $\AC:\SMAT^n \to \Real^{m}$ is a linear map defined by $\AC(\X) = (\A_1 \bullet \X, \A_2 \bullet \X ,\ldots, \A_m \bullet \X)^\top$ with $\A_1,\A_2,\dots,\A_m \in \SMAT^n$, and
the vector $\b \in \Real^m$
in \eqref{primal} is a given vector.
The symbol  
$\X \succ \O$ ($\X \succeq \O$) for a matrix $\X \in \SMAT^n$ denotes that 
$\X$ is a positive definite 
(positive semidefinite, respectively) matrix.

Many log-determinant SDP problems can be rewritten as the form of \eqref{primal}. A model \eqref{primal} with $H =1$ and $\QC_1(\X)$ being the vector of elements of $\X$ is equivalent to the following problem:
\begin{equation}\label{logdetsdp}
\begin{split}
 \min_{\X\in\SMAT^n} & \ \ f(\X) := \C \bullet \X - \mu \log \det \X + \lambda \sum_{i = 1}^n \sum_{j = 1}^n |X_{ij}| \\
\mbox{s.t.}  & \ \  \AC(\X) = \b, \X \succ \O.
\end{split}
\end{equation}
Furthermore, when $\lambda = 0$ and the linear constraint is $X_{ij} = 0$ for $(i,j) \in \Omega \subseteq \{ (i,j) \ | \ 1 \le i < j \le n \}$, \eqref{logdetsdp} turns into the Gaussian graphical models \cite{lauritzen1996graphical}, which corresponds to the graphical interpretation of sparse covariance selection model \cite{dempster1972covariance}.

The model \eqref{primal} also covers the following problem in 
Lin et al.~\cite{lin2020estimation} that estimates sparse Gaussian graphical models with hidden clustering structures.
The fourth term in the objective function induces a clustering structure
of the concentration matrix.
\begin{equation}\label{hidden}
\begin{split}
 \min_{\X\in\bS^n} & \ \ f(\X) := \C \bullet \X - \mu \log \det \X + \rho \sum_{i < j}|X_{ij}| + \lambda \sum_{i < j}\sum_{s < t}|X_{ij} - X_{st}|\\
\mbox{s.t.}  & \ \  \AC (\X) = \b, \, \X \succ \O,
\end{split}
\end{equation}
% This model is also able to be transformed into a structure of our model by defining the function $\QC_{ij, st}(\X):= X_{ij} - X_{st}$ and setting the degree of all norms to $1$, that is $p_{ij, st} = 1$.
% 
Another important model in \eqref{primal} is the block $\ell_{\infty}$-regularized log-likelihood minimization problem in Duchi et al.~\cite{duchi2008projected} to estimate sparsity between entire blocks of variables:
\begin{align}
	 \min_{\X \in \SMAT^n} & \ \ \f(\X) := \C \bullet \X - \log \det \X + \sum_{k = 1}^{K} \lambda_{k} \max \{ |X_{ij}| \big| (i,j) \in G_{k} \},\label{duchi}
\end{align}
where entries in $\X$ are divided into disjoint subsets $G_1 , G_2 , \dots , G_K (K < n^2)$. The last term in \eqref{duchi} is the group $\ell_{\infty}$-regularized covariance selection used to enforce the sparsity between blocks. Other than this regularization, there is also the group $\ell_1$ and $\ell_2$-regularized regression, see \textit{,e.g.,} \cite{bach2008consistency,yuan2006model}. 
Extensions of model \eqref{duchi} have also been examined. Honorio et al.~\cite{honorio2010multi} proposed the multi-task structure learning problem for Gaussian graphical models that consider promoting a consistent sparseness pattern across arbitrary tasks by using the regularizer to penalize corresponding edges across the task. Yang et al.~\cite{yang2013proximal} discussed a model that replaces the last term in the objective function in \eqref{duchi} with $\ell_p$-norm for $p \in \{1, 2, \infty\}$. % We can easily see that these models can be covered by the model \eqref{primal} in this paper.

When the regularizer is restricted to the $\ell_1$-norm \eqref{logdetsdp}, many methods have been proposed. Wang et al.~\cite{WST2010} proposed a Newton-CG primal proximal point algorithm, % which is a primal proximal point algorithm with the newton-CG algorithm to solve the inner subproblem
and Li et al.~\cite{li2010inexact} modified an inexact primal-dual path-following interior-point algorithm to solve the log-det SDP with a large number of linear constraints. Hsie et al.~\cite{hsieh14a} proposed a quadratic approximation for sparse inverse covariance estimation (QUIC) based on the Newton method and a quadratic approximation. 
Wang et al.~\cite{Wang2016} generated an initial point of the algorithm by using a proximal augmented Lagrangian method and then computed the accurate solution by applying a Newton-CG augmented Lagrangian method. % to solve the subproblem that the objective function is the summation of log-determinant and convex quadratic terms. 

Nakagaki et al.~\cite{nakagaki2020dual} proposed a dual spectral projected gradient method (DSPG) for solving the dual problem of \eqref{logdetsdp}.
The method is an iterative method that uses an inexact projection to avoid the difficulty of computing the orthogonal projection to the intersection of two convex sets, while still having the advantages of the spectral projected gradient (SPG) method \cite{birgin2000nonmonotone}. In particular, an important advantage here is that it requires only the function values and the first-order derivatives, making it faster than other second-order derivatives methods. 
However, their convergence analysis heavily depended on the $\ell_1$ norm in the objective function of \eqref{logdetsdp}.
In \eqref{primal} that we address in this paper, we do not assume a specific structure of the linear map $\QC_h$  and also the number of the terms $H$ that corresponds to the number of variables in the dual problem, therefore, we cannot simply apply \cite{nakagaki2020dual} to the generalized problem \eqref{primal}.

In this paper, we extend the DSPG method to deal with the general form of log-determinant optimization problems~\eqref{primal}. % Based on the DSPG, we propose a numerical method for solving the dual problem of \eqref{primal}. 
To enhance the efficiency, we apply a similar reformulation technique as in \cite{namchaisiri2024new}.
We embed the $\ell_p$-norm structure in the objective function
in \eqref{primal} into constraints so that the objective is differentiable, and we combine the projection onto the constraints related to the $\ell_p$-norm. 

In this paper, our main contributions are as follows.
\begin{itemize}
\item We propose the generalized model \eqref{primal}, which covers many log-det models.
\item We develop a numerical method (Algorithm \ref{alg}) for solving \eqref{primal}, and present the convergence analysis of the algorithm.
\item We show the efficiency of the proposed method with numerical experiments (Section \ref{sec:numex}). For \eqref{primal} with the matrix dimension $n=2000$, the method in Duchi et al.~\cite{duchi2008projected} demanded 163.48 seconds while the proposed method consumed only 65.88 seconds to attain the same solution accuracy.
\end{itemize}

The remainder of this paper is organized as follows. We describe the structure of the dual problem and the proposed DSPG-based method (Algorithm \ref{alg}) in Section~\ref{sec:method} and discuss the convergence analysis in Section~\ref{sec:convergence} focusing on the generalized part of \eqref{primal}. In Section~\ref{sec:numex}, we present the result of numerical experiments on the log-likelihood minimization problem, block constraint problem, and multi-task structure. Finally, we conclude in Section~\ref{sec:conclusion}.

\subsection{Notation and symbols}\label{subsec:prelim}

Let $\|\X\|: = \sqrt{\X \bullet \X}$ denote the
Frobenius norm for a matrix $\X\in\bS^n$.
We also use $\|\y\|$ to represent the Euclidean norm of the vector $\y \in \Real^n$, that is, $\|\y\| := \|\y\|_2$.

Given a linear map $\AC$. We denote the adjoint operator of $\AC$ as $\AC^\top$. 
We define the operator norm of $\AC :  \SMAT^n \to \Real^m$ and $\AC^\top : \Real^m \to \SMAT^n$ as $\|\AC\| := \sup_{\X \ne \O} \left\{\frac{\|\AC(\X)\|}{\|\X\|}\right\}$ and $\|\AC^\top\| := \sup_{\bm y \ne \0} \left\{\frac{\|\AC^\top(\bm y)\|}{\|\bm y\|}\right\}$, respectively.

Let $\UC$ be the direct product space $\Real^m \times \SMAT^{n} \times \dots \times \SMAT^{n}$. We define the inner product of $\U_1 = (\y_1,\S^1_1,\dots,\S^1_H) \in \UC$ and $\U_2 = (\y_2,\S^2_1,\dots,\S^2_H) \in \UC$ by
%\begin{equation*}
$ \langle\U_1,\,\U_2\rangle := \bm y_1^\top\bm y_2 + \bm S^1_1\bullet\bm S^2_1 + \dots +  \bm S^1_H\bullet\bm S^2_H$.
% \end{equation*}
We also define the norm of $\U \in \UC$ as $||\U||:= \sqrt{\langle\U,\,\U\rangle}$.

For $p \geq 1$ and $\lambda > 0$, we define the $\ell_p$-ball with radius $\lambda$ as
\begin{equation*}
    \BC_p^{\lambda} := \{\x \ | \ ||\x||_{p} \leq \lambda\}.
\end{equation*}

We use $P_{\SC}(\cdot)$ to denote the projection onto the convex set $\SC$, i.e.,
\begin{equation*}
P_{\SC}(\cdot) := \arg\min_{\X\in\Omega}\|\X - \cdot\|.
\end{equation*}

\section{The proposed method}\label{sec:method}

Let $\QC_h^{\top}: \Real^{n_h} \to \Real^m$ be the adjoint operators of $\QC_h$. To extend the DSPG method developed in \cite{nakagaki2020dual} for the dual problem of \eqref{logdetsdp}, 
we need the dual problem of our problem \eqref{primal}:
\begin{equation}\label{dual1}
\begin{split}
         \max_{\y \in \Real^m,\, \z_h \in \Real^{n_h} (h=1,\dots,H) } & \ \  \b^{\top}\y + \mu \log \det \left(\C - \AC^{\top}(\y) + \sum_{h=1}^{H} \QC_h^{\top}(\z_h)\right) +  n \mu - n \mu \log \mu\\
        \mbox{s.t.} & \ \ \|\z_h\|_{p_h^*} \leq \lambda_h \ (h=1,\dots,H), \, \C - \AC^{\top}(\y) + \sum_{h=1}^{H} \QC_h^{\top}(\z_h) \succ \O,
\end{split}
\end{equation}
where $\|\cdot\|_{p_h^*}$ is the dual norm of $\|\cdot\|_{p_h}$ such that $\frac{1}{p_h} + \frac{1}{p_h^*} = 1$.
More precisely, $p_h^*$ for $p_h \in [1, \infty]$ is given by
\begin{align}
  p_h^* = \begin{cases}
    \infty & \text{if} \ p_h = 1 \\
    p_h / (p_h - 1) & \text{if} \ 1 < p_h < \infty \\
    1 & \text{if} \ p_h = \infty.
  \end{cases}
\end{align}

To increase the computational efficiency, we employ a similar approach as in Namchaisiri et al.~\cite{namchaisiri2024new}. We introduce sets $\SC_h = \{ \S \ \in \SMAT^n | \ \S = \QC_h^{\top}(\z), \z \in \BC_{p_h^*}^{\lambda_h} \}$ for $h = 1,2,\dots,H$. Denoting the variables $(\y, \S_1, \dots, \S_H) \in \UC$ as one composite variable $\U$, we can rewrite \eqref{dual1} as follows:
\begin{equation}\label{dual}
\begin{split}
         \max_{\U \in \UC} & \ \  g(\U) := \b^T\y + \mu \log \det \left(\C - \AC^{\top}(\y) + \sum_{h=1}^{H} \S_h\right) +  n \mu - n \mu \log \mu\\
        \mbox{s.t.} & \ \ \S_h \in \SC_h (h=1,\dots,H), \,
\C - \AC^{\top}(\y) + \sum_{h=1}^{H} \S_h \succ \O.
\end{split}
\tag{$\DC$}
\end{equation} 
Note that the difficulty due to the $\ell_{p_h}$-norm in the problem is embedded into the set $\SC_h$. 

For the convergence analysis in the following section, we employ the same assumption as \cite{nakagaki2020dual,namchaisiri2024new}:
\begin{ASSM}\label{assp}
We assume the following statements hold for problem \eqref{primal}
and its corresponding dual \eqref{dual}:
\begin{itemize}
\item[{\rm (i)}] The set of matrices $\A_1, \A_2 , \A_m$ is linearly independent. In the other words, $\AC$ is surjective;
\item[{\rm (ii)}] There exists a strictly feasible solution $\widehat{\X} \succ \O$ of primal problem~\eqref{primal} that satisfies $\AC (\widehat{\X}) = \b$;
\item[{\rm (iii)}] There exists a composite variable $\widehat{\U}$ that satisfies the constraint of dual problem~\eqref{dual}. 
\end{itemize}
\end{ASSM}
%\todo{We had better avoid the dual problem of \eqref{logdetsdp} and its notation since this paragraph is long and it reduces the importance of the contribution in this paper.}
%\begin{equation}\label{dual_dspg}
%\begin{split}
% \max_{\bm y\in\Real^m,\,\bm W\in\bS^n} & \ \  \b^\top\bm y   + \mu\log\det\left(\C - \A^\top(\bm y) + \bm W \right) + n \mu - n \mu \log \mu \\
%\mbox{s.t.}  & \ \ \|\bm W\|_{\infty} \le \rho\\
%& \ \ \C - \A^\top (\bm y) + \bm W \succ\O.
%\end{split}
%\end{equation}
%We use $\|\bm W\|_{\infty} := \max \{ |W_{ij}| \}$ to denote the maximum absolute element of $\bm W$, and 
%$\A^\top : \R^m \to \S^n$  the adjoint operator of $\A$ (that is, $\A^\top (\y) := \sum_{i=1}^m \A_i y_i$). In this method, instead of computing the direct projection to the feasible region of \eqref{dual_dspg}, it computes the projection onto $\WC:= {\W \in \SMAT^n | \|\W\|_{\infty}}$, and use the step length to generate the point that satisfies $\C - \A^\top (\y) + \W \succ\O$. The advantage of this method is the computation of projection to $\WC$ is much easier than the projection to the feasible region and makes it an efficient algorithm, which is shown in the numerical experiments in \cite{nakagaki2020dual}. 
Let $\FC$ be the feasible region of \eqref{dual}. 
We express this feasible region as the intersection 
$\FC = \MC \cap \NC$ of  $\MC := \{\U \in \UC \ | \ \y \in \Real^m, \S_h \in \SC_h (h=1,\dots,H) \}$ and $\NC := \{ \U \in \UC \ | \ \C - \AC^{\top}(\y) + \sum_{h=1}^{H} \S_h \succ \O\}$. Let $\X(\U) := \mu \left(\C - \AC^{\top}(\y) + \sum_{h=1}^{H} \S_h\right)^{-1}$. Thus, the gradient of $g(\U)$ can be expressed as
\begin{align*}
% &\\
\nabla g(\U) =& \left(\nabla_{\y} g(\U),  \nabla_{\S_1} g(\U) ,\, \dots ,\,  \nabla_{\S_H} g(\U) \right) \\
% =&\Bigg(\b - \mu \AC\left(\left(\C - \S - \AC^{\top}(\y)\right)^{-1}\right), -\mu\left(\C -\S - \AC^{\top}(\y)\right)^{-1} ,\, \dots ,\, -\mu\left(\C -\S - \AC^{\top}(\y)\right)^{-1}\Bigg) \\
=&\left(\b - \AC(\X(\U)), \X(\U) ,\, \dots ,\, \X(\U) \right).
\end{align*}
%, which is defined by
% \begin{equation}\label{Projdef}
% \begin{split}
%  \arg \min_{\U \in \SC } & \ \  \frac{1}{2} ||\U - \hat{\X}||^2.
% \end{split}
% \end{equation}

We introduce a map
\[
\BC(\U) := - \AC^{\top}(\y) + \sum_{h=1}^{H} \S_h , \ \ \mbox{where}\ \ \U = (\y, \S_1 , \, \dots \, , \S_H),
\]
and define $\U^k := \left(\y^k, \S_1^k, \dots, \S_H^k\right)$. % Since the projections onto $\Real^m, \SC_1,\dots, \SC_H$ are independent of each other, we decompose the projection $P_{\MC}(\U)$ as $(\y, P_{\SC_1}(\S_1), \dots, P_{\SC_H}(\S_H))$. 
For solving the dual problem~\eqref{dual}, we propose Algorithm~\ref{alg} below by extending 
the DSPG method in \cite{namchaisiri2024new}. The step length of Step 3 is the Barzilai-Borwein step \cite{barzilai1988two}. The advantages of this step length are mentioned in \cite{HAGER08,nakagaki2020dual}, in particular, the linear convergence can be obtained under a mild condition. % without the line search for solving the unconstrained problem if the initial point is in a neighborhood of the local minimizer and the Hessian matrix of the objective function is positive definite.
\begin{algorithm}[H]
\caption{A DSPG algorithm for generalized log-det semidefinite programming}
\label{alg}
\begin{description}

\item [Step 0.] Choose  parameters $\varepsilon > 0, \tau \in (0,1), \gamma \in (0,1),  0 <\beta< 1, 0 < \alpha_{\min} < \alpha_{\max} < \infty$ and  integer $M > 0$. Take $\U^0 \in \FC$ and $\alpha_0 \in [\alpha_{\min}, \alpha_{\max}]$. Set the iteration number $k = 0$.

 \item [Step 1.] Let $\Delta \U^k_{(1)} := \left([\Delta\y^k]_{(1)}, [\Delta \S_{1}^k]_{(1)},\,\dots ,\, [\Delta\S_{H}^k]_{(1)} \right) := P_{\MC}\left(\U^k + \nabla g(\U^k)\right)  - \U^k$. If $\|\Delta \U^k_{(1)}\| \le \varepsilon$, terminate;  otherwise, go to \textbf{Step 2}.

\item [Step 2.]   Let $\D^k := \left(\Delta\y^k, \Delta\S_{1}^k,\dots, \Delta\S_{H}^k\right) := P_{\MC}\left(\U^k + \alpha_k\nabla g(\U^k)\right) - \U^k$. Let $\bm L_k$ be the Cholesky factorization of $\C + \BC(\U^k)$, that is $\bm L_k\bm L_k^\top= \C + \BC(\U^k)$, and $\theta$ be the minimum eigenvalue of $\bm L_k^{-1}\BC(\D^k) \left(\bm  L_k^\top\right)^{-1}$. Set
\begin{equation*}
 \nu_k :=
\begin{cases}
  1 & \text{if} \ \theta \ge 0,\\
 \min\{1,\, -\tau/\theta \}  & \text{otherwise}.
\end{cases}
\end{equation*}
Apply a line search to find the largest element $\sigma_k \in\{1,\,\beta,\,\beta^2, \ldots\}$ such that
 \begin{equation} 
g(\U^k + \sigma_k \nu_k\D^k) \geq  \min\limits_{[k - M + 1]_+ \le l \le k} g(\U^l) + \gamma\sigma_k \nu_k\langle\nabla g(\U^k),\,\D^k\rangle.
\end{equation}

\item [Step 3.]  Let $\U^{k+1} = \U^k + \sigma_k \nu_k\D^k$. Let $p_k:= \langle\U^{k+1} - \U^k,\, \nabla g(\U^{k+1}) - \nabla g(\U^k)\rangle$. Set
\begin{equation*}
 \alpha_{k+1} :=
\begin{cases}
  \alpha_{\max} & \text{if} \ p_k \ge 0,\\
\min\left\{\alpha_{\max},\,\max\left\{\alpha_{\min},\, -\|\U^{k+1} - \U^k\|^2/p_k \right\}\right\} & \text{otherwise}.
\end{cases}
\end{equation*}
Set $k \leftarrow k + 1$. Return to \textbf{Step 1}.

\end{description}

\end{algorithm}

From the viewpoint of the convergence analysis,
Algorithm \ref{alg} converges for any linear map $\QC_h$ % $\QC_h^{\top}$ 
as shown in Section \ref{sec:convergence}. This is not proved in the previous DSPG papers~\cite{nakagaki2020dual,namchaisiri2024new}.
On the other hand, 
the efficiency of Algorithm \ref{alg} depends on the computation of the projection $P_{\MC}(\cdot)$. 
The objective functions of the numerical experiments in Section~\ref{sec:numex} were chosen so that the projection onto each $\SC_h$ can be computed within appropriate computation costs.

\section{Convergence analysis}\label{sec:convergence}

We show the convergence of Algorithm~\ref{alg} to the optimal value by extending the analysis in \cite{namchaisiri2024new}. 
In particular, we focus on the proof of the part $\sum_{h=1}^{H} \S_h$ that corresponds to $\sum_{h=1}^{H} \lambda_h \|\QC_h(\X)\|_{p_h}$ in the primal generalized log-det SDPs~\eqref{primal}. 

For the convergence proof below, we will show the validity of the stopping criterion in Lemma~\ref{LEMM:optimality-condition}, the lower bound of step length that prevents the premature termination in Lemma~\ref{LEMM:steplength-bound}, and the convergence of the output sequence of Algorithm \ref{alg} to the optimal solution in Theorem~\ref{theo:dspg}. Let $\{\U^k\}$ be the sequence generated by Algorithm~\ref{alg}. We use the notation $\X^k:= \X(\U^k)$ in the proof.

We start with Lemma~\ref{LEMM:bounded} that shows the feasibility and the boundedness of the sequence generated from Algorithm~\ref{alg}.
We introduce a level set
\begin{equation*}
\mathcal{L} :=  \left\{\U \in \FC : \ g(\U) \ge g(\U^0)\right\}.
\end{equation*}
% The proof of Lemma~\ref{LEMM:bounded} directly follows the discussion 
% of \cite{namchaisiri2024new}, thus we omit its proof.

\begin{LEMM}\label{LEMM:bounded}
  {\upshape \cite[Lemma~3.1]{namchaisiri2024new}}$\{\U^k\}\subseteq\mathcal{L}$ and $\{\U^k\}$ is bounded. 
\end{LEMM}
In particular, the surjectivity of $\AC$ and the existence of the primal interior feasible point $\widehat{\X}$ in Assumption~\ref{assp} play an essential part in the proof of \cite{namchaisiri2024new}.

With the boundedness of $\SC_1, \dots, \SC_H$, we also obtain the following lemma, which will be used in Lemma~\ref{LEMM:mainobjectiveconv}.
\begin{LEMM}\label{LEMM:levelset-bounded}
  {\upshape \cite[Lemma~3.3]{nakagaki2020dual}}The level set $\mathcal{L}$ is bounded. 
\end{LEMM}
The proof of this lemma also utilizes 
the surjectivity of $\AC$ to derive the boundedness 
of the component of $\y$ in $\U$.

Using Lemma~\ref{LEMM:bounded} and strict concavity of $\log \det$ term in the objective function of \eqref{dual}, we can show that the eigenvalues
of $\C + \BC(\U^k)$ are sandwiched with a positive lower bound and a finite upper bound.
Since $\X^k = \mu (\C + \BC(\U^k))^{-1}$, 
we obtain 
%we obtain Corollary~\ref{CORO:bound-on-X}, which indicates 
the boundedness of the sequence $\{\X^k\}$.
\begin{CORO}\label{CORO:bound-on-X}
{\upshape \cite[Remark 3.5]{nakagaki2020dual}}
There exist bounds $\beta_{\X}^{\min}$ and $\beta_{\X}^{\max}$ such that $\O \preceq \beta_{\X}^{\min} \I \preceq \X^k \preceq \beta_{\X}^{\max} \I$ for all $k$. 
Thus $\eta_{\X}:=\sqrt{n} \beta_{\X}^{\max}$ is also a bound for $\{\X^k\}$ such that $\|\X^k\| \le \eta_{\X}$ for all $k$. 
% A set $\left\{ \X^k \right\}$ % $\left\{ \X(\U^k)\right\}$ 
% is bounded. 	
\end{CORO}
Furthermore, we can also show the boundedness of other components related to $\U^k$. % of $\D^k$.

\begin{LEMM}\label{LEMM:bound}
There exist bounds 
$\eta_{\X^{-1}}$ and $\eta_{ \Delta\bm y} > 0$ such that 
% $\|\X^k\| \le \eta_{\X}$, 
$\|(\X^k)^{-1}\| \le \eta_{\X^{-1}}$ and $\| \Delta \y^k \| \le \eta_{\Delta \bm y}$ for all $k$. There also exists
 $\eta_{\Delta \bm S} > 0$ that satisfies
 $\| \Delta \bm S_h^k\| \le \eta_{\Delta \bm S}$ for all $h = 1, 2, \dots, H$ and all $k$. 
\end{LEMM}

\begin{proof}
    We can obtain $\eta_{\X^{-1}}$ and $\eta_{\Delta \y}$ by applying Lemma~3.4 in \cite{namchaisiri2024new}. For $h = 1, 2, \dots, H$, since $\U^k \in \MC \subset \FC$ from Lemma~\ref{LEMM:bounded}, we have $\bm S_h^k \in \SC_h$ for any $k\ge 1$. Therefore, due to the property of the projection, it holds that $\| \Delta \bm S_h^k\| = \left\| P_{\SC_h}\left(\bm S_h^k + \alpha_k \X^k\right) - \bm S_h^k \right\| \leq \| \alpha_k \X^k\|\leq \alpha_{\max} \eta_{\X} =: \eta_{\Delta \bm S}$.
\end{proof}

The following lemma derives an optimality condition 
for the dual problem \eqref{dual} from the viewpoint 
of the projection, and this guarantees the validity of the stopping criterion in Step 1 of Algorithm~\ref{alg}. The proofs in \cite{namchaisiri2024new}
cannot directly give this lemma, since the structure of set $\MC$ now includes multiple sets $\SC_1, \ldots, \SC_H$.
% The following lemma shows the optimality condition of  Algorithm~\ref{alg}. This lemma guarantees that the solution that we obtain from Algorithm~\ref{alg} is an optimal solution of \eqref{dual}.

\begin{LEMM}\label{LEMM:optimality-condition}
$\U^* \in \FC$ is an optimal solution of the dual problem~\eqref{dual} if and only if there exists $\alpha > 0$ such that
\begin{equation}
% \P_{\MC}((\y^*,\S_1^*,\S_2^*, \dots, \S_j^*) + \alpha \nabla g (\y^*,\S_1^*,\S_2^*, \dots, \S_j^*)) = (\y^*,\S_1^*,\S_2^*, \dots, \S_j^*).
P_{\MC}(\U^* + \alpha \nabla g(\U^*)) = \U^*.
\end{equation}

\end{LEMM}

\begin{proof}

We can see that the objective equation is equivalent to

\begin{equation}
\begin{split}
\U^* = \arg \min \ \left\{\frac{1}{2} \| \U - (\U^* + \alpha \nabla g(\U^*))\|^2 + \delta_{\MC} (\U)\right\}.
\end{split}\label{eq:condition-U}
\end{equation}
where $\delta_{\MC}$ is the indicator function of $\MC$. 

Let $\U^*$ be decomposed into $\U^* = (\y^*,\S_1^*,\S_2^*, \dots, \S_H^*)$.
From the definition of the projection onto $\S_h$ for each $h = 1,\dots,H$, the equality
$P_{\S_h}(\S_h^* + \alpha \nabla_{\S_h} g(\U^*)) = \S_h^*$ holds 
if and only if
\begin{equation}
\begin{split}
\S_h^* &= {\arg \min}_{\S_h \in \SC_h} \ \frac{1}{2} \| \S_h - (\S_h^* + \alpha \nabla_{\S_h} g(\U^*))\|^2 \\
&= \arg \min \ \left\{\frac{1}{2} \| \S_h - (\S_h^* + \alpha \nabla_{\S_h} g(\U^*))\|^2 + \delta_{\SC_h} (\S_h)\right\}.
\end{split}\label{eq:condition-S}
\end{equation}
From the definition $\MC := \{\U \in \UC \ | \ \y \in \Real^m, \S_h \in \SC_h (h=1,\dots,H) \}$, we can see that projection computation for each component of $\U$ to $\MC$ is independent from the other components.

Since the subgradient of the indicator function is a normal cone, \eqref{eq:condition-U} is equivalent to $0 \in \alpha \nabla g (\U^*) + N_{\MC} (\U^*)$, where $N_{\MC} (\U^*)$ denotes the normal cone of $\MC$ at $\U^*$. Since $\NC = \left\{ \U \in \UC \ | \ \C - \AC^{\top}(\y) + \sum_{h=1}^H \S_h \succ \O \right\}$, we can see that $\NC$ is an open set, which means all of the elements in $\NC$ are interior points of $\NC$. According to Theorem 3.30 in \cite{Nonsmooth2014}, we obtain that $N_{\MC} (\U^*) = N_{\MC \ \cap \ \NC}(\U^*) = N_{\FC} (\U^*)$. This implies $0 \in \alpha \nabla g (\U^*) + N_{\FC} (\U^*)$, and 
this condition is equivalent to the optimality of $\U^*$ in the dual problem~\eqref{dual}.
\end{proof}

We will show in Lemma~\ref{LEMM:Subbound} 
that the difference between $\lambda_h\|\QC_h(\X^k)\|_{p_h}$ in the primal objective function  and an inner product $ \S^k_h \bullet \X^k$ can be estimated with $\|[\Delta \S^k_h]_{(1)}\|$.
This difference is a part of the duality gap between \eqref{primal}
and \eqref{dual} in the $k$th iteration. Therefore, we employ the limit
$\liminf_{k\to \infty} \|[\Delta \S^k_h]_{(1)}\| \to 0 \ $ (which will be shown in Lemma~\ref{LEMM:liminf-delta}) to show the limit $\liminf_{k\to \infty}|g(\U^k) - g^*|$ in Lemma~\ref{LEMM:mainobjectiveconv} which leads to the convergence to the objective value.

\begin{LEMM}\label{LEMM:Subbound}
For $h = 1,2,\dots,H$, $|\lambda_h\|\QC_h(\X^k)\|_{p_h} - \S^k_h \bullet \X^k|$ is bounded by $\|[\Delta \S^k_h]_{(1)}\|$. 
More precisely, 
\[
|\lambda_h\|\QC_h(\X^k)\|_{p_h} - \S^k_h \bullet \X^k| \leq c_h\|[\Delta \S^k_h]_{(1)}\|
\]
holds for all $k$
with $c_h := \eta_{\X} + \lambda_h\sqrt{n_h}(||\QC_h|| + ||\QC_h^\top||)$.
\end{LEMM}

\begin{proof}
Since $\U^k \in \MC$, we know $\S^k_h \in \SC_h$, thus there exists $\z^k_h \in \Real^{n_h}$ such that $\|\z^k_h\|_{p_h^*} \leq \lambda_h$ and $\S^k_h = \QC_h^{\top}(\z^k_h)$. Therefore, we have 
\[
\S^k_h \bullet \X^k = \QC_h^{\top}(\z_h^k) \bullet \X^k = \QC_h(\X^k)^{\top} \z_h^k.
\]

The Holder's inequality $|\a^T \b| \leq \|\a\|_{p^*} \|\b\|_{p}$ holds for any vectors $\a,\b$ with the same length and $p \geq 1$. Thus,  $|\S_h^k \bullet \X^k| = |\QC_h(\X^k)^{\top} \z_h^k | \leq ||\z_h^k|| _{p_h^*} ||\QC_h(\X^k)||_{p_h} \leq \lambda_h||\QC_h(\X^k)||_{p_h}$, and this leads to $\lambda_h||\QC_h(\X^k)||_p - \S_h^k \bullet \X^k \geq 0$. Furthermore, we know that $||\QC_h(\X)|| _{p_h} \leq ||\QC_h(\X)||_1$ for $p_h \ge 1$, thus it holds that 
\begin{equation}\label{bound1}
 |\lambda_h \|\QC_h(\X^k)\|_{p_h} - \S^k_h \bullet \X^k| \leq | \lambda_h \|\QC_h(\X^k)\|_{1} - \S^k_h \bullet \X^k|.
\end{equation}

Let $\hat{\S}_h^k = P_{\SC_h}(\S_h^k + \X^k)$. 
Then there exists $\hat{\z}_h^k \in \Real^{n_h}$ such that $\|\hat{\z}_h^k\|_{p^*} \leq \lambda_h$ and $\hat{\S}_h^k = \QC_h^{\top}(\hat{\z}_h^k)$. Let $\tilde{\e}_h^k \in \Real^{n_h}$ be a vector whose elements are the signs of $\QC_h(\X^k)$. We have
\begin{align}
 |\lambda_h \|\QC_h(\X^k)\|_{1} - \S^k_h \bullet \X^k| 
= & \ |\lambda_h(\tilde{\e}_h^k)^{\top} \QC_h(\X^k) - \S^k_h \bullet \X^k| 
= \ |\lambda_h \QC_h^{\top} (\tilde{\e}_h^k) \bullet \X^k - \S^k_h \bullet \X^k| \\ 
\leq & \ 
| \lambda_h \QC_h^{\top} (\tilde{\e}_h^k) \bullet \X^k - \hat{\S}_h^k \bullet \X^k| + 
|(\hat{\S}_h^k - \S_h^k) \bullet \X^k|  \\
\le & \ 
|(\lambda_h \QC_h^{\top} (\tilde{\e}_h^k) - \hat{\S}_h^k) \bullet \X^k| + 
\|[\Delta \S_h^k]_{(1)} \| \cdot \|\X^k\|.
\label{bound2}
\end{align}

We can see that the second term of \eqref{bound2} is bounded by $\|[\Delta \S_h^k]_{(1)}\|$ due to $\|\X^k\| \le \eta_{\X}$ in Lemma~\ref{LEMM:bound}.  Therefore, our focus here is the first term. Due to properties P1 in {\cite[Proposition 2.1]{HAGER08}} as a property of the projection $\hat{\S}_h^k = P_{\SC_h}(\S_h^k + \X^k)$, we can derive an inequality
\begin{align}
  \left(\X^k + [\Delta \S_h^k]_{(1)}\right) \bullet \left(\lambda_h \QC_h^{\top} (\tilde{\e}_h^k) - \hat{\S}_h^k\right) = \left(\S_h^k + \X^k - \hat{\S}_h^k\right) \bullet \left(\lambda_h \QC_h^{\top} (\tilde{\e}_h^k) - \hat{\S}_h^k\right)
\leq 0.
\end{align}
This indicates
\begin{align}\label{subbound2.5}
\X^k \bullet \left(\lambda_h \QC_h^{\top} (\tilde{\e}_h^k) - \hat{\S}_h^k\right) \leq [\Delta \S_h^k]_{(1)} \bullet \left(\lambda_h \QC_h^{\top} (\tilde{\e}_h^k) - \hat{\S}_h^k\right).
\end{align}
We show the nonnegativity of $\X^k \bullet \left(\lambda_h \QC_h^{\top} (\tilde{\e}_h^k) - \hat{\S}_h^k\right) = \X^k \bullet \left(\lambda_h \QC_h^{\top} (\tilde{\e}_h^k) - \QC_h^{\top}(\hat{\z}_h^k)\right) = \QC(\X^k) \bullet (\lambda_h \tilde{\e}_h^k - \hat{\z}_h^k) = \sum_{j = 1}^{n_h} [\QC_h(\X^k)]_{j} [\lambda_h \tilde{\e}_h^k - \hat{\z}_h^k]_j.$
For each $j$, we can see that the sign of $[\lambda_h \tilde{\e}_h^k - \hat{\z}_h^k]_j$ is the same as $\tilde{\e}_h^k$ because $|[\hat{\z}_h^k]_j| \leq \|{\hat{\z}}^k\|_{p_h^*} \leq \lambda_h$. Since $\tilde{\e}_h^k$ is the sign of $[\QC_h(\X^k)]_{j}$, we obtain $[\QC_h(\X^k)]_{j} [\lambda_h \tilde{\e}_h^k - \hat{\z}_h^k]_{j} \geq 0$, hence,
\begin{align}
    \X^k \bullet \left(\lambda_h \QC_h^{\top} (\tilde{\e}_h^k) - \hat{\S}_h^k\right) = 
    \sum_{j = 1}^{n_h} [\QC_h(\X^k)]_{j} [\lambda_h \tilde{\e}_h^k - \hat{\z}_h^k]_j
    %\QC_h(\X^k)^{\top} (\tilde{\e}_h^k + \hat{\z}_h^k) 
    \geq 0.
\end{align}
Applying this result into \eqref{subbound2.5} and using Lemma~\ref{LEMM:bound}, we obtain
\begin{align}\label{bound3}
|\X^k \bullet \left(\lambda_h \QC_h^{\top} (\tilde{\e}_h^k) - \hat{\S}_h^k\right)| \leq |[\Delta \S_h^k]_{(1)} \bullet \left(\lambda_h \QC_h^{\top} (\tilde{\e}_h^k) - \hat{\S}_h^k\right)|.
\end{align}
Since $\hat{\S}_h^k \in \SC_h$, we obtain the bound
\begin{align}\label{bound4}
||\hat{\S}_h^k|| & \leq ||\QC_h|| ||\hat{\z}_h^k|| \leq \sqrt{n_h}||\QC_h|| ||\hat{\z}_h^k||_{\infty} \leq \sqrt{n_h}||\QC_h|| ||\hat{\z}_h^k||_{p_h^*} \leq \sqrt{n_h} \lambda_h ||\QC_h||.
\end{align}

Using \eqref{bound4} in \eqref{bound3}, we have
\begin{align}\label{bound5}
    |\X^k \bullet \left(\lambda_h \QC_h^{\top} (\tilde{\e}_h^k) - \hat{\S}_h^k\right)| \leq (\lambda_h\sqrt{n_h}||\QC^{\top}_h|| + \lambda_h\sqrt{n_h}||\QC_h||)\|[\Delta \S_h^k]_{(1)}\|.
\end{align}

This indicates that $|\X^k \bullet \left(\lambda_h \QC_h^{\top} (\tilde{\e}_h^k) - \hat{\S}_h^k\right)|$ is bounded by $\|[\Delta \S_h^k]_{(1)}\|$.
Combining \eqref{bound1}, \eqref{bound2} and \eqref{bound5}, we can conclude this lemma with the value of $c_h = \eta_{\X} + \lambda_h\sqrt{n_h}(||\QC_h|| + ||\QC^{\top}_h||)$. 
\end{proof}

% Before showing the next proof, We mention the following two lemmas, which can apply to the Algorithm \ref{alg}. 

Lemma \ref{LEMM:steplength-bound} shows the lower bound of the step length, which prevents the Algorithm \ref{alg} from terminating before the stopping criterion is satisfied.

\begin{LEMM}\label{LEMM:steplength-bound}
There exists a bound $(\sigma\nu)_{\min} > 0$ such that step length $\sigma_k\nu_k > (\sigma\nu)_{\min}$ for all $k$.
\end{LEMM}

\begin{proof}
    Let $\BC(\D^k) := \AC^{\top}(\Delta \y) + \sum_{h=1}^{H} \Delta \S_h$. Using Lemma~\ref{LEMM:bound}, we have its bound 
    \begin{align}
        ||\BC(\D^k)|| & \leq \ ||\AC^{\top}||||\Delta \y|| + \sum_{h=1}^{H} ||\Delta \S_h||  \leq ||\AC^{\top}||\eta_{\Delta \y} + H\eta_{\Delta \S} \\
        & \leq \ \sqrt{H+1}\max\{||\AC^{\top}||,1\}||\D^k||.
        \label{eq:bound-BD}
    \end{align}

    We divide the proof into two sections. Firstly, we show that $\nu_k$ has a lower bound, and then we will show that $\sigma_k \nu_k$ has a lower bound.

    From the definition of $\nu_k$ in Step 2 of Algorithm~\ref{alg}, we consider only the case that $\theta < 0$. The definition of $\theta$ that is the minimum eigenvalue of $\L_k^{-1} \BC(\D^k) (\L_k^{\top})^{-1}$ implies that $\theta$ is the maximum value that satisfies
    %\begin{equation}
        $\L_k^{-1} \BC(\D^k) (\L_k^{\top})^{-1} \succeq \theta \I$.
    %\end{equation}
    Using the property $\L_k \L_k^{\top} = \C + \BC (\U^k)$ and $\theta < 0$, % we can use the linear transformation to show that 
		$\nu' = -1/\theta$ is the maximum value that satisfies
    \begin{equation}
        \C + \BC (\U^k) + \nu' \BC(\D^k) \succeq \O. 
    \end{equation}

    On the other hand, using the bound from Corollary~\ref{CORO:bound-on-X}, we have
    \begin{align}
        & \ \C + \BC (\U^k) + \nu' \BC(\D^k)  = \ \mu (\X^k)^{-1} + \nu'\BC(\D^k) \\
			\succeq  & \ \frac{\mu}{\beta_{\X}^{\max}} \I  - \nu'||\BC(\D^k)|| \I 
        \succeq \left(\frac{\mu}{\beta_{\X}^{\max}} - \nu'\left(||\AC^{\top}||\eta_{\Delta \y} + H\eta_{\Delta \S}\right)\right) \I.
    \end{align}
    Therefore, if we consider the interval $0 \leq \nu' \leq \mu / \left(\beta_{\X}^{\max}(||\AC^{\top}||\eta_{\Delta \y} + H\eta_{\Delta \S})\right) =: \nu_{\min}$, we can guarantee the positive semidefiniteness $\C + \BC (\U^k) + \nu' \BC(\D^k) \succeq 0$. %This makes $-1/\theta \geq \nu_{\min}$, and
		This implies 
		a positive lower bound $\nu_k \ge \min\{1,\tau \nu_{\min}\}$. 
		% that $\nu_k = \min\{1,-\tau/\theta\}$ has a lower bound.
    
From the discussion of {\upshape \cite[Lemma 3.5]{namchaisiri2024new}}, we have the bound
\begin{align}
  \left\|\left(\X(\U^k + \nu \bm D^k) - \X(\U^k) \right)\right\|  \leq \frac{\nu }{\mu\left(\frac{1-\tau}{\beta_{\X}^{\max}}\right)^2}\|\BC(\bm D^k)\|.
\end{align}
Therefore, it holds for $\lambda \ge 0$ that 
    \begin{align}
    & \ \|\nabla g(\U^k + \lambda \D^k) - \nabla g(\U^k)\|\\
    = & \ \left\|\left(-\AC(\X(\U^k + \lambda \D^k) - \X(\U^k)), \X(\U^k + \lambda \D^k) - \X(\U^k), \dots , \X(\U^k + \lambda \D^k) - \X(\U^k)\right)\right\|\\
    \leq & \ \sqrt{\|\AC\|^2 + H}\ \|\X(\U^k + \lambda \D^k) - \X(\U^k)\| \\
    \leq & \frac{\sqrt{\|\AC\|^2 + H}}{\mu\left(\frac{1-\tau}{\beta_{\X}^{\max}}\right)^2}\lambda \|\BC(\bm D^k)\| 
    \leq  \left(\frac{\sqrt{\|\AC\|^2 + H}}{\mu\left(\frac{1-\tau}{\beta_{\X}^{\max}}\right)^2} \sqrt{H+1}\max\{||\AC^{\top}||,1\}\right) \lambda ||\D^k|| = \lambda L ||\D^k||, 
    \label{eq:bound-L}
    \end{align}
    where the last inequality was due to \eqref{eq:bound-BD}
    and we introduce 
     $L := \frac{\sqrt{\left\|\AC\right\|^2 + H}}{\mu(\frac{1-\tau}{\beta_{\X}^{\max}})^2} \sqrt{H+1}\max\{||\AC^{\top}||,1\}$.

    We focus on the termination condition of the non-monotone Armijo rule in the line search of Step~2 in Algorithm~\ref{alg}. 
    If it terminates at $\sigma_k = 1$, then we obtain a lower bound as $\sigma_k \nu_k \ge  \min\{1,\tau \nu_{\min}\}$. 
    If it terminates at $\sigma_k = \beta^j$ for some $j \ge 1$, this indicates that the termination condition is not satisfied at $\sigma_k = \beta^{j-1}$. This further implies 
    \begin{align}\label{j-1condition} 
    g(\U^k + \beta^{j-1} \nu_k \D^k) & <  \min\limits_{[k - M + 1]_+ \leq l \leq k} g(\U^l) + \gamma\beta^{j-1}\nu_k\langle\nabla g(\U^k),\,\D^k\rangle \\
    & \leq g(\U^k) + \gamma\beta^{j-1}\nu_k\langle\nabla g(\U^k),\,\D^k\rangle.
    \end{align}
		Therefore, we have
    \begin{align}\label{taylor}
        \gamma \beta^{j-1}\nu_k\langle\nabla g(\U^k),\,\D^k\rangle
        \geq & \ \  g(\U^k + \beta^{j-1}\nu_k\D^k) - g(\U^k) \\
        = & \ \  \beta^{j-1}\nu_k \langle \nabla g(\U^k) , \D^k \rangle
        + \int_0^{\beta^{j-1}\nu_k} \langle \nabla g(\U^k + \lambda \D^k )- \nabla g(\U^k) , \D^k \rangle d \lambda.\\
        \geq & \ \ \beta^{j-1}\nu_k \langle \nabla g(\U^k) , \D^k \rangle - \frac{L(\beta^{j-1}\nu_k)^2}{2} \|\D^k\|^2,
    \end{align}
    where the equality is due to Taylor's expansion and the last inequality holds from \eqref{eq:bound-L}. If $||\D^k|| = 0$,  Algorithm \ref{alg} should be terminated in Step 1, so we can here assume $||\D^k|| > 0$.
    From $\beta > 0$ and $\nu_k \ge \min\{1,\tau \nu_{\min}\}$, 
    we know that 
    $\beta^{j-1} \nu_k > 0$.
   Therefore, \eqref{taylor} is equivalent to
    \begin{align}\label{boundj-1k}
        \beta^{j-1}\nu_k \geq \frac{2(1 - \gamma)}{L}\frac{\langle \nabla g(\U^k) , \D^k \rangle}{||\D^k||^2}. 
    \end{align}
    Using a property of the projection (see \cite[Proposition 2.1, P1]{HAGER08}), it holds that 
    \begin{equation*}
    \begin{split}
     \langle  (\U^k + \alpha_k\nabla g(\U^k)) - P_{\MC}(\U^k + \alpha_k\nabla g(\U^k)),\, \U^k - P_{\MC}(\U^k + \alpha_k\nabla g(\U^k))\rangle \le 0.
    \end{split}
	\end{equation*}
	The left-hand side is 
	\begin{equation*}
    \begin{split}	 
    \langle- \D^k + \alpha_k \nabla g(\U^k) ,\, - \D^k\rangle = - \alpha_k\langle\nabla g(\U^k),\,\D^k\rangle + \|\D^k\|^2,
    \end{split}
    \end{equation*}
    thus we obtain  
    \begin{equation}
     \frac{\langle\nabla g(\U^k),\,\D^k\rangle}{\|\D^k\|^2} \geq \frac{1}{\alpha_k} \geq \frac{1}{\alpha_{\max}}.\label{eq:bound-D}
    \end{equation}
    Applying \eqref{eq:bound-D} to \eqref{boundj-1k}, we obtain the positive lower bound $\beta^{j-1}\nu_k \geq \frac{2(1 - \gamma)}{\alpha_{\max}L}$, which leads to the positive lower bound $\sigma_k\nu_k := \beta^{j}\nu_k \geq \frac{2\beta(1 - \gamma)}{\alpha_{\max}L}$.
    Combining the case $\sigma_k = 1$, we obtain 
    $\sigma_k \nu_k \ge \min\{1, \tau \nu_{\min}, \frac{2\beta(1 - \gamma)}{\alpha_{\max}L}\}$ for all $k$. 
    This completes the proof.    
\end{proof}

Lemma~\ref{LEMM:liminf-delta} indicates the limit of the search direction of Algorithm \ref{alg}, which will be used in the proof of Lemma~\ref{LEMM:mainobjectiveconv}.

\begin{LEMM}\label{LEMM:liminf-delta}
Algorithm~\ref{alg} with the stopping criterion parameter $\varepsilon = 0$ terminates after reaching the optimal value $g^*$, or it generates the sequence $\{\U^k\}$ that satisfies
\[
\liminf_{k\to \infty}\|\Delta \U^k_{(1)}\| = 0.
\]
\end{LEMM}

\begin{proof}
  As discussed in {\cite[Lemma 3.6]{namchaisiri2024new}}, the dual objective value increases at least in every $M$ iteration, where $M$ is used in the line search of Step 2 and it has an upper bound since the level set $\mathcal{L}$ is bounded from Lemma~\ref{LEMM:levelset-bounded}. Using the existence of the lower bound of $\sigma_k \nu_k$ from 
  Lemma~\ref{LEMM:steplength-bound}, 
  we can show $\liminf_{k\to \infty}\|\D^k\| = 0$. 
  Furthermore, from the properties P4 and P5 in {\cite[Proposition 2.1]{HAGER08}}, we can show that the inequality $\|\Delta \U^k_{(1)}\| \leq \max \{1,\alpha_{\max}\} \|\D^k\|$ holds. This implies the statement of this lemma. % Combining
\end{proof}

We now discuss the convergence of the objective value to the optimal value of \eqref{dual}, denoted with $g^*$.
\begin{LEMM}\label{LEMM:mainobjectiveconv}
Algorithm~\ref{alg} with the stopping criterion parameter $\varepsilon = 0$ terminates after reaching the optimal value $g^*$, or it generates the sequence $\{\U^k\}$ that satisfies
\[
\liminf_{k\to \infty}|g(\U^k) - g^*| = 0.
\]
\end{LEMM}

\begin{proof}
Let $\X^*$ be the optimal solution of \eqref{primal}. Due to the strict convexity of the logarithm function, $\X^*$ is the unique solution. We decompose the difference of $|g(\U^k) - g^*|$
into the summation of three parts by the following inequality:
\begin{equation}\label{divideparts}
|g(\U^k) - g^*|  \leq |g(\U^k) - f(\X^k)| + |f(\X^k) - f(\X^*)| + |f(\X^*) - g^*|.
\end{equation}
The first part of \eqref{divideparts} can be evaluated as 
\begin{align}
    |g(\U^k) - f(\X^k)| =& \ \Bigg| \b^T\y^k + \mu \log \det \left(\C - \AC^{\top}(\y^k) + \sum_{h=1}^H \S_h^k\right) +  n \mu - n \mu \log \mu \\
    & \  - \C \bullet \X^k + \mu \log \det \X^k - \sum_{h=1}^H \|\QC_h(\X^k)\|_{p_h}\Bigg| \\
    = & \ \left| (\y^k)^{\top} (\b - \AC(\X^k)) - \sum_{h = 1}^{H}(\lambda_h \|\QC(\X^k)\|_{p} - \S_h^k \bullet \X^k)\right| \\
    \leq & \ \|\y^k\| \|\AC\| \|\X^* - \X^k\| + \sum_{h=1}^{H} \left|\lambda_h\|\QC(\X^k)\|_{p} - \S_h^k \bullet \X^k \right|.
\end{align}

The first term is bounded by $\|\X^k - \X^*\|$ due to the boundedness of $\y^k$ from Lemma~\ref{LEMM:bound}, and the second summation is bounded by $\|[\Delta \S^k_h]_{(1)}\|$ as shown in Lemma~\ref{LEMM:Subbound}. From \cite[Lemma 3.15]{nakagaki2020dual}, $\|\X^k - \X^*\|$ can be bounded by $||\Delta \U^k_{(1)}||$. Therefore, $|g(\U^k) - f(\X^k)|$ is bounded by $||\Delta \U^k_{(1)}||$. We can use the same discussion as \cite[Lemma 3.8]{namchaisiri2024new} to show that the second term of \eqref{divideparts} is also bounded by $\|\Delta \U^k_{(1)}\|$. Due to Assumption~\ref{assp}, the duality theorem between \eqref{primal} and \eqref{dual} holds, and makes the third term in \eqref{divideparts} zero.
Combining three terms, we can show that $|g(\U^k) - g^*|$ is bounded by $||\Delta \U^k_{(1)}||$. Finally, from Lemma~\ref{LEMM:liminf-delta}, we have $\liminf_{k\to \infty} ||\Delta \U^k_{(1)}|| = 0$. This completes the proof.
\end{proof} 

Combining these results, we can show the main convergence of Algorithm~\ref{alg}.

\begin{THEO}\label{theo:dspg}
Algorithm~\ref{alg} with the stopping criterion parameter $\varepsilon = 0$ terminates after reaching the optimal value $g^*$, or it generates the sequence $\{\U^k\}$ that satisfies
\[
\lim_{k\to \infty}|g(\U^k) - g^*| = 0.
\]
\end{THEO}

\begin{proof}
    
    We use the contradiction to prove this theorem. Suppose that there exists $\epsilon > 0$ and an infinite increasing sequence $\{k_1, k_2, \dots\}$ such that $g(\U^{k_i}) \le g^* - \epsilon$ for all $i$ and $g(\U^{l}) >  g^* - \epsilon$ for all $l \notin \{k_1, k_2, \dots\}$.

		Firstly, we will show that $k_{i+1} - k_{i} \leq M$ for all $i$. Suppose that there exists an $i$ such that $k_{i+1} - k_{i} > M$, which means all elements in $\{k_{i+1} - M , k_{i+1} - M + 1 , \dots  k_{i+1} - 1\}$ is not contained in the sequence of $k_i$. Therefore,  $g(\U^{l})  >  g^* - \epsilon$ for all $k_{i+1} - M \leq l \leq k_{i+1} - 1$, and this is equivalent to $g(\U^{l}) > g^* - \epsilon$.

    From \eqref{eq:bound-D}, we obtain  
    \begin{equation}
     \alpha_k\langle\nabla g(\U^k),\,\D^k\rangle \geq \|\D^k\|^2 \geq 0.   \label{eq:bound-D-1}
    \end{equation}
    
    Applying this to the inequality in Step 2 of Algorithm \ref{alg}, we obtain that
    \begin{align}
        g(\U^{k_{i+1}}) \geq \min\limits_{k_{i+1} - M \leq l \leq k_{i+1} - 1} g(\U^l) \geq g^* - \epsilon.
    \end{align}
    However, this contradicts to $g(\U^{k_i}) \le g^* - \epsilon$ for all $i$.
		Therefore, we obtain $k_{i+1} - k_{i} \leq M$ for all $i$.
    
    From the proof in Lemma~\ref{LEMM:mainobjectiveconv}, 
    we know that $|g(\U^k) - g^*|$ is bounded by $||\Delta \U^k_{(1)}||$. This means the sequence $||\Delta \U^{k_i}_{(1)}||$ has a lower bound $\bar{\epsilon}$. In addition, the proof of Lemma~\ref{LEMM:liminf-delta} employed $\|\Delta \U^k_{(1)}\| \leq \max \{1,\alpha_{\max}\} \|\D^{k}\|$. This leads to the existence of the lower bound of $\|\D^{k_i}\| = \bar{\epsilon} 
 / \max \{1,\alpha_{\max}\} =: \delta$. From \eqref{eq:bound-D}, 
we know $\left\langle\nabla g(\U^k),\,\D^k\right\rangle 
 \geq \frac{\left\|\D^k\right\|^2}{\alpha_k} \geq \frac{\delta^2}{\alpha_{\max}}$.
 Therefore, we derive
    \begin{align}
        g(\U^{k_{i}}) \geq \min\limits_{k_{i} - M \leq l \leq k_{i} - 1} g(\U^l) + \bar{\delta},
    \end{align}
    where $\bar{\delta}:= \gamma(\sigma \nu)_{\min} \frac{\delta^2}{\alpha_{\max}}$. Let $l(k)$ be an integer such that $k - M \leq l(k) \leq k - 1$ and $g(\U^{l(k)}) = \min\limits_{k_ - M \leq l \leq k - 1} g(\U^l)$. Since $k_{i} - k_{i-1} \leq M$, we have $g(\U^{l(k_i)}) \leq g(\U^{k_{i-1}}) \leq g^* - \epsilon$. This means $l(k_i)$ is in the sequence $\{k_1, k_2, \dots\}$. Therefore, for each $k_i > M$, there exists $k_j$ such that
	$	k_i - k_j \le M$ and 
    \begin{align}
        g(\U^{k_{i}}) \geq g(\U^{k_{j}}) + \bar{\delta}.
    \end{align}
		Hence, 
		if $\{k_1, k_2, \dots\}$ is an infinite sequence,
		we know  $g(\U^{k_{i}}) \to \infty$ when we take $i \to \infty$, 
		and this contradicts the existence of the optimal solution $g^*$. This completes the proof.
\end{proof}

\section{Numerical experiments}\label{sec:numex}

In this section, we present numerical results of Algorithm \ref{alg} on a problem of the form of \eqref{primal} with synthetic data. In the second experiment that performs on a problem with block regularization, we compare Algorithm \ref{alg} with the projected gradient (PG) method proposed in Duchi et al.~\cite[Algorithm~3]{duchi2008projected}. All experiments in this section were conducted in Matlab R2022b on a 64-bit PC with Intel Core i7-7700K CPU (4.20 GHz, 4 cores) and 16 GB RAM.

For Algorithm \ref{alg}, we set the parameters $\tau = 0.5 , \gamma = 10^{-3}, \beta = 0.5, \alpha_{\min} = 10^{-8}, \alpha_{\max} = 10^8$ and $M = 5$. For PG, we set the parameters $\alpha = 0.5$, $\beta = 0.5$. We take an initial point of Algorithm \ref{alg} and PG as $\U^0 = (\y^0 , \S^0_1 , \dots , \S^k_H) = (\0, \O, \dots, \O)$ and $W = \O$, respectively. We set the stopping criterion of Algorithm~\ref{alg} as
$\|\Delta \U^k_{(1)}\| \leq \varepsilon$ with $\varepsilon = 10^{-12}$.
The limit of the iterations is 5000 iterations, and the computation time limit is 7200 seconds.

For the projection onto the $\BC_{p_h^*}^{\lambda_h}$,
we use the direct computation when $p_h^* \in \{1, 2, \infty\}$:
\begin{align}
	[P_{\BC_{p_h^*}^{\lambda}}(\z)]_i 
	= \left\{
	\begin{array}{lcl}
 z_i - (\text{sign}(z_i)\min\{|z_i|,s\}) & \text{for} & p_h^* = 1 \\
 \frac{\lambda z_i}{\max\{||\z||_2,\lambda\}} & \text{for} & 
 p_h^* = 2 \\
\max\{ -\lambda , \min \{ \lambda , z_i\}\} & \text{for} & p_h^* = \infty.
	\end{array}
	\right.
\end{align}
where $s$ for the case $p_h^* = 1$ is a real number that satisfies $\sum_{i = 1}^{N} \max\{0,|x_i|-s\} = \lambda$. 
For $p_h^* \notin \{1, 2, \infty\}$,
we use the Newton method implemented in \texttt{bpdq\_proj\_lpball}\footnote{\url{https://wiki.epfl.ch/bpdq/documents/help/bpdq_toolbox/common/bpdq_proj_lpball.html}}.

% Let $\x \in \Real^N$. We compute the projection to $\BC_1^{\lambda}$ by
% \begin{align}
%     [P_{\BC_1^{\lambda}}(\x)]_i := x_i - (\text{sign}(x_i)\min\{|x_i|,s\}),
% \end{align}
% where $s$ is a real number that satisfies $\sum_{i = 1}^{N} \max\{0,|x_i|-s\} = \lambda$. For the projection to $\BC_2^{\lambda}$ and $\BC_{\infty}^{\lambda}$, we use
% \begin{align}
%     [P_{\BC_2^{\lambda}}(\x)]_i & := \frac{\lambda x_n}{\max\{||\x||_2,\lambda\}} \\
%     [P_{\BC_{\infty}^{\lambda}}(\x)]_ i& := \max\{ -\lambda , \min \{ \lambda , x_n\}\}.
% \end{align}

To evaluate the performance of the algorithm, we use %the number of iterations, the computation time, and 
the relative gap defined in \cite{namchaisiri2024new} as
\[
\text{Gap} = \frac{|P - D|}{\max\{ 1 , (|P| + |D|) / 2\}},
\]
where $P$ and $D$ are the output values of primal and dual objective functions, respectively.

\subsection{Log-likelihood minimization problem with \texorpdfstring{$\ell_{p}$}{lp}-norm extension}\label{subsec:fbnormloglikelihood}

In this experiment, we evaluate the efficiency of the proposed method by solving the following synthetic problem:
\begin{align}\label{primalexp1-1}
 \min_{\X\in\bS^n} & \ \ \C \bullet \X - \mu \log \det \X + 
 \sum_{h=1}^{H}\lambda_h \left(\sum_{1 \leq i < j \leq n} |X_{ij}|^{p_h}\right)^{\frac{1}{p_h}}\\
\mbox{s.t.}  & \ \  X_{ij} = 0 \ \forall (i,j) \in \Omega,  \X \succ \O.
\end{align}
Here, $\Omega \subset \{(i,j): 1 \le i \le j \le n\}$ is a set that defines the linear constraints.
We introduce a linear map $vect: \SMAT^{n} \to \Real^{\frac{n(n-1)}{2}}$ that reshapes a $n \times n$ symmetric matrix into a $\frac{n(n-1)}{2}$-dimensional vector by stacking the column vectors in the upper triangular part of the matrix.
We can rewrite the problem into the form of \eqref{primal} as follows:
\begin{equation}\label{primalexp1-2}
\begin{split}
 \min_{\X\in\bS^n} & \ \ f(\X) := \C \bullet \X - \mu \log \det \X + \sum_{h=1}^{H}\lambda_h ||vect(\X)||_{p_h}\\
\mbox{s.t.}  & \ \  X_{ij} = 0 \ \forall (i,j) \in \Omega,  \X \succ \O.
\end{split}
\end{equation}

To generate the input matrix $\C$ and the set $\Omega$, 
we used the same procedure in \cite{nakagaki2020dual}. Firstly, we randomly generated a $n \times n$ sparse positive matrix $\Sigma^{-1}$ with a density parameter $\sigma = 0.1$, and constructed the covariance matrix $\C \in \SMAT^n$ from $\max\{2n,2000\}$ samples of the multivariate Gaussian distribution $\NC(0,\Sigma)$.
We made a set $\Omega^{'} := \{(i,j) \ | \  \Sigma^{-1}_{ij} = 0,  |i - j| > 5, 1 \le i < j \le n \}$, and randomly selected a half of entries in $\Omega^{'}$ to be $\Omega$. % linear constraints 

% Let $p_1 , \dots , p_H$ be a value in the interval $[1,\infty]$. 

Firstly, we conducted experiments on $H = 1$, which corresponds to the following problem:
\begin{align}\label{primalexp1-h1}
 \min_{\X\in\bS^n} & \ \ \C \bullet \X - \mu \log \det \X + 
 \lambda_1||vect(\X)||_{p_1}\\
\mbox{s.t.}  & \ \  X_{ij} = 0 \ \forall (i,j) \in \Omega,  \X \succ \O.
\end{align}
Here, we set $\lambda_1 = 0.001 \cdot n^{1-\frac{1}{p_1}}$.

 Table~\ref{table:exp1-1} shows the numerical results of problem \eqref{primalexp1-h1}. The first column is the size $n$ and the number of constraints $|\Omega|$. The second, third, and fourth columns are the number of iterations, the computation time in seconds, and the relative gap for Algorithm~\ref{alg}. %, where
% The relative gap is defined by
%\[
% $\text{Gap} = \frac{|P - D|}{\max\{ 1 , (|P| + |D|) / 2\}}$,
%\]
%with $P$ and $D$ being the output objective values of primal and dual problems, respectively. 
The other six columns are for the different values of $p_1$.
\begin{table}[H] \centering \ra{1.3}
	\caption{Numerical results on $\ell_{p_1}$-norm log-likelihood minimization problem ($H = 1$).}
		\label{table:exp1-1} 
		\addtolength{\tabcolsep}{-2pt}
	\scalebox{0.8}{
		\begin{tabular}{@{}rrrrrrrrrrrr@{}}\toprule &\multicolumn{3}{c}{$p_1 = 1$} & \phantom{abc} &\multicolumn{3}{c}{$p_1 = 2$} & \phantom{abc} &\multicolumn{3}{c}{$p_1 = \infty$}\\ 
				\cmidrule{2-4} \cmidrule{6-8} \cmidrule{10-12}  \multicolumn{1}{c}{$(n , |\Omega|)$} & Iterations & Time & Gap && Iterations & Time & Gap && Iterations & Time & Gap\\
		$(500,56086)$ & 207 & 15.03 & 3.70e-7 && 241 & 17.58 & 7.23e-7 && 193 & 15.62 & 1.66e-7\\ 
		$(1000,220647)$ & 169 &  57.75 & 2.36e-7 && 187 & 63.95 & 2.30e-7 && 153 & 61.40 & 1.46e-7\\
		$(2000,859795)$ & 127 & 292.69 & 7.05e-7  && 202 & 468.78 & 5.75e-7 && 155 & 442.88 & 2.55e-7\\
		$(4000,3311218)$ & 123 &  2104.99 & 5.66e-7 && 217 &  3662.50 & 6.02e-6 && 196 & 3300.02 & 3.09e-7\\
		\bottomrule 
		\end{tabular}}
		\end{table}

 We can see that Algorithm~\ref{alg} can solve the problem \eqref{primalexp1-h1} in different values of $p_1$. The original DSPG method in Nakagaki et al.~\cite{nakagaki2020dual} can solve only the case of $p_1 = 1$, while Table~\ref{table:exp1-1} indicates Algorithm~\ref{alg} can solve other $p_1 > 1$ with enough accuracy. 
 
 Table~\ref{table:exp1-1-1} reports a more detailed computation time of the projection of problem \eqref{primalexp1-h1}. The second column is the computation time of projection of all the iterations, the third column is the average time of projection per iteration,
 and the fourth column is the average of the entire computation per iteration. %computing in the same way. 
The other six columns are for $p_1=2$ and $p_1 = \infty$. We can observe from Table~\ref{table:exp1-1-1} that even the projection time for $p_1 = \infty$ is higher than others due to the complexity of the projection onto the $\ell_{p_1^*} = \ell_{1}$-ball, the average times per iteration of the three experiments are not much different. This is because the computation cost of each projection is much lower than the cost of Cholesky factorization in Step~2 which demands $O(n^3)$ operations.

 \begin{table}[H] \centering \ra{1.3}
\caption{Projection time on $\ell_{p_1}$-norm log-likelihood minimization problem ($H = 1$).}
	\label{table:exp1-1-1} 
	\addtolength{\tabcolsep}{-2pt}
\scalebox{0.7}{
	\begin{tabular}{@{}rrrrrrrrrrrr@{}}\toprule &\multicolumn{3}{c}{$p_1 = 1$} & \phantom{abc} &\multicolumn{3}{c}{$p_1 = 2$} & \phantom{abc} &\multicolumn{3}{c}{$p_1 = \infty$}\\ 
			\cmidrule{2-4} \cmidrule{6-8} \cmidrule{10-12}  \multicolumn{1}{c}{$(n , |\Omega|)$} & Time.Proj & Avg.Proj. & Avg.Comp. && Time.Proj & Avg.Proj. & Avg.Comp. && Time.Proj & Avg.Proj. & Avg.Comp.\\
	$(500,56086)$ & 0.27 & 1.31e-3 & 7.26e-2 && 0.34 & 1.41e-3 & 0.07 && 2.79 & 1.44e-2 & 8.09e-2\\ 
	$(1000,220647)$ & 0.87 &  5.19e-3 & 0.34 && 0.97 & 5.20e-3 & 0.34 && 6.94 & 4.54e-2 & 0.40\\
	$(2000,859795)$ & 2.31 & 1.81e-2 & 2.30  && 4.31 & 2.13e-2 & 2.32 && 31.1 & 0.20 & 2.57\\
	$(4000,3311218)$ & 9.45 &  7.69e-2 & 17.11 && 18.48 &  8.52e-2 & 18.88 && 165.22 & 0.84 & 16.84\\
	\bottomrule 
	\end{tabular}}
	\end{table}

Next, we did the experiments with $H = 2$:
\begin{align}\label{primalexp1-h2}
 \min_{\X\in\bS^n} & \ \ \C \bullet \X - \mu \log \det \X + \lambda_{p_1}||vect(\X)||_{1} + \lambda_{2}||vect(\X)||_{p_2}\\
\mbox{s.t.}  & \ \  X_{ij} = 0 \ \forall (i,j) \in \Omega,  \X \succ \O.
\end{align}
Similar to the previous problem, we set $\lambda_1 = 0.001 \cdot n^{1-\frac{1}{p_1}}$ and $\lambda_2 = 0.001 \cdot n^{1-\frac{1}{p_2}}$. This problem has different norms in the objective function so we can evaluate that Algorithm~\ref{alg} can handle the problem by summating the extension structure. We mention that the existing DSPG methods cannot apply to this problem.

Table \ref{table:exp1-2} reports the numerical results on problem \eqref{primalexp1-h2}. 
Similarly to the previous experiment, the computation time is almost proportional to the number of iterations, thus the average time for each iteration does not vary so much. For example, the numerical experiments at $n = 4000$ takes $20.03$ seconds per iteration for the case $(p_1,p_2) = (1,2)$,   $18.06$ seconds per iteration for the case $(p_1,p_2) = (1,\infty)$, and $17.54$ seconds per iteration for the case $(p_1,p_2) = (2,\infty)$.

\begin{table}[H] \centering \ra{1.3}
\caption{Numerical results on $\ell_p$-norm log-likelihood minimization problem ($H = 2$).}
	\label{table:exp1-2} 
	\addtolength{\tabcolsep}{-2pt}
\scalebox{0.8}{
	\begin{tabular}{@{}rrrrrrrrrrrr@{}}\toprule &\multicolumn{3}{c}{$(p_1,p_2) = (1,2)$} & \phantom{abc} &\multicolumn{3}{c}{$(p_1,p_2) = (1,\infty)$} & \phantom{abc} &\multicolumn{3}{c}{$(p_1,p_2) = (2,\infty)$}\\ 
			\cmidrule{2-4} \cmidrule{6-8} \cmidrule{10-12}  \multicolumn{1}{c}{$(n , |\Omega|)$} & Iterations & Time & Gap && Iterations & Time & Gap && Iterations & Time & Gap\\
	$(500,56086)$ & 283 & 30.86 & 6.67e-7 && 202 & 19.36 & 9.25e-8 && 210 & 18.70 & 2.36e-7\\ 
	$(1000,220647)$ & 191 &  97.29 & 5.50e-7 && 164 & 78.58 & 2.69e-7 && 150 & 60.26 & 2.13e-8\\
	$(2000,859795)$ & 120 & 351.24 & 6.35e-7  && 124 & 348.99 & 2.45e-7 && 166 & 412.15 & 5.45e-7\\
	$(4000,3311218)$ & 105 &  2103.68 & 1.04e-6 && 76 & 1372.87 & 1.19e-6 && 212 & 3718.64 & 2.63e-6\\
	\bottomrule 
	\end{tabular}}
	\end{table}

We further conducted an experiment with the value $p_h$ that requires high costs for the projection to $\ell_{p_h^*}$-ball. We examined cases with $H = 1$, $p_1 = 3/2$ and $4/3$, which need the projection on $\ell_3$-ball and $\ell_4$-ball respectively. 

Table \ref{table:exp1-4} 
% and \ref{table:exp1-5} 
shows the numerical results of problem \eqref{primalexp1-h1} with $p_1 = 3/2$ and $4/3$ respectively. Now we observe that the computation time for the projection is very high compared to the previous experiments with $p_1 = 1, 2, \infty$. In all pairs of $(n,|\Omega|)$, the computation time of the projection occupies more than half of the entire computation time. This makes Algorithm \ref{alg} cannot solve the experiments with $(n,|\Omega|) = (4000,3311218)$ in the computation time limit of 7200 seconds. % From this result, we can observe that the efficiency of Algorithm~\ref{alg} is low when the projection is heavier than the Cholesky factorization.

\begin{table} \centering \ra{1.3}
	\caption{Numerical results on $\ell_{3/2}$-norm and $\ell_{4/3}$-norm log-likelihood minimization problems.}
		\label{table:exp1-4} 
		\addtolength{\tabcolsep}{-2pt}
  \scalebox{0.9}{
		\begin{tabular}{@{}rrrrrrrrr@{}}\toprule 
  {$(n , |\Omega|)$} & Iterations & Time & Time.Proj & Avg.Comp & Avg.Proj & Gap \\
	\hline
	\multicolumn{7}{c}{$p=\frac{3}{2}$} \\
	\hline
		$(500,56086)$ & 337 & 95.42 & 72.92 & 0.28 & 0.22 & 8.32e-7\\ 
		$(1000,220647)$ & 229 & 292.18 & 207.29 & 1.28 & 0.91 & 8.98e-7 \\
		$(2000,859795)$ & 166 & 953.22 & 563.64 & 5.74 & 3.40 & 1.53e-6 \\
		$(4000,3311218)$ &  &  & OOT &  &   &  \\
		\hline
		\multicolumn{7}{c}{$p=\frac{4}{3}$} \\
		\hline
			$(500,56086)$ & 344 & 147.16 & 124.02 & 0.42 & 0.36 & 2.19e-6\\ 
		$(1000,220647)$ & 343 & 606.59 & 493.98 & 1.44 & 1.77 & 1.84e-6\\
		$(2000,859795)$ & 387 & 3103.23 & 2218.91 & 8.01 & 5.73 & 2.66e-6 \\
		$(4000,3311218)$ &  &   & OOT  &  &   &  \\
		\bottomrule 
		\end{tabular}}
		\end{table}
  % \begin{table}[H] \centering \ra{1.3}
	% \caption{Numerical results on $\ell_{4/3}$-norm log-likelihood minimization problem.}
	% 	\label{table:exp1-5} 
	% 	\addtolength{\tabcolsep}{-2pt}
  % \scalebox{0.9}{
	% 	\begin{tabular}{@{}rrrrrrrrr@{}}\toprule 
  % {$(n , |\Omega|)$} & Iterations & Time & Time.Proj & Avg.Comp & Avg.Proj & Gap\\
	% 	$(500,56086)$ & 344 & 147.16 & 124.02 & 0.42 & 0.36 & 2.19e-6\\ 
	% 	$(1000,220647)$ & 343 & 606.59 & 493.98 & 1.44 & 1.77 & 1.84e-6\\
	% 	$(2000,859795)$ & 387 & 3103.23 & 2218.91 & 8.01 & 5.73 & 2.66e-6 \\
	% 	$(4000,3311218)$ &  &   & OOT  &  &   &  \\
	% 	\bottomrule 
	% 	\end{tabular}}
	% 	\end{table} 

\subsection{Block \texorpdfstring{$\ell_{\infty}$}{linfty}-regularized log-likelihood minimization problem}\label{subsec:blocklinftyloglikelihood}

In this experiment, 
we present the numerical results on  
the following block $l_{\infty}$-regularized log-likelihood minimization problem \eqref{duchi} from Duchi et al.~\cite{duchi2008projected}: % with the specific block construction as follows :
\begin{align}
	 \min_{\X \in \SMAT^n} & \ \ \C \bullet \X - \log \det \X + 
   % \sum_{h_1=1}^{H} \sum_{h_2=h_1}^{H}
   \sum_{h_1,h_2 = 1}^{H}
   \lambda_{h_1 h_2} \max \{ |X_{ij}| \big| (i,j) \in G_{h_1 h_2} \}.\label{primalexp2-1}
\end{align}

We constructed the covariance matrix $\C \in \SMAT^n$ as % similarly to the experiment 
in Section~\ref{subsec:fbnormloglikelihood}.
The sets $G_{11}, \dots, G_{H_1H_2}$ are the disjoint subsets of 
$\{1,\dots,n\} \times \{1,\dots,n\}$
that partition the inverse covariance matrix $\BSigma$.
More precisely, to inherit the symmetry of $\BSigma \in \SMAT^n$, 
we first divide the set of row/column indexes 
$\{1, \dots, n\}$ into $\bar{G}_1, \bar{G}_2, \dots,\bar{G}_H$
and define $G_{h_1 h_2} = 
(\bar{G}_{h_1} \times \bar{G}_{h_2}) \cup (\bar{G}_{h_2} \times \bar{G}_{h_1})$
for $1 \leq h_1 \leq h_2 \leq H$.
  
We can rewrite \eqref{primalexp2-1} into the form of \eqref{primal} by defining a linear map $\QC_{h_1 h_2}(\X)$ that stacks
$\{X_{ij} \big| (i,j) \in G_{h_1 h_2} \}$ as a vector
and taking its infinity norm (that is, $p_{h_1 h_2} = \infty$).
We set the penalty $\lambda_{h_1 h_2} := \rho |G_{h_1 h_2}|$
where $\rho$ is a positive parameter 
and $|G_{h_1 h_2}|$ is the cardinality of $G_{h_1 h_2}$
so that $\lambda_{h_1 h_2}$ is proportional to $|G_{h_1 h_2}|$.
In the following result, we used $\rho = 0.001$. 

We compare the performance of Algorithm~\ref{alg} with that of the PG method. 
PG uses a stopping criterion based on the KKT condition in \cite{li2010inexact}
\[
  \max\left\{\frac{|P-D|}{1 + |P| + |D|},pinf,dinf\right\} \leq \textrm{gaptol},
  \]
  where $\textrm{gaptol} = 10^{-6}$. 
  The values of $pinf$ and $dinf$ are the residuals of the constraints in the primal and dual problems, respectively, as in \cite{li2010inexact}. For the comparison,
we include Algorithm~\ref{alg} that also employs the same stopping criterion denoted as Algorithm~\ref{alg} (KKT). 
  % More precisely, let
  % \begin{align*}
  % pinf & = \frac{\AC(\X) - \b}{1 + ||\b||}, 
  % \end{align*}
  We note that 
  $pinf = 0$ holds in this experiment
  since the problem  is unconstrained,
  and $dinf = 0$ holds in Algorithm~\ref{alg} because the generated sequence $\{\U^k\}$ is always dual feasible. 
  % The same projection algorithm was used in both algorithms. 
    All algorithms can obtain the result with the relative gap around $10^{-6}$ with these settings.

  Table~\ref{table:exp2} shows the performance comparison between the two algorithms on \eqref{primalexp2-1}. Algorithm~\ref{alg} works well in the large instance. %, and still performs better even if we change the stopping criterion. 
  We can also see that the number of iterations that Algorithm~\ref{alg} takes is less than PG. The convergence efficiency of Algorithm~\ref{alg} will be clearer for larger instances. % when the inner algorithm demands a high computation cost. 
	We can see from the experiment with $(n,k) = (6000,50)$ that Algorithm \ref{alg} can give the estimated solution of \eqref{primalexp2-1} in $1602$ seconds, but PG cannot solve it in the computation time limit of 7200 seconds.

\begin{table}\centering \ra{1.3}
	\caption{Numerical results on block $\ell_{\infty}$-regularized log-likelihood minimization problem (OOT is out of time)}
	\label{table:exp2}  
	\addtolength{\tabcolsep}{-1pt}
\scalebox{0.8}{
	\begin{tabular}{@{}rrrrrrrrrrrr@{}}\toprule &\multicolumn{3}{c}{Algorithm~\ref{alg}} & \phantom{abc} &\multicolumn{3}{c}{Algorithm~\ref{alg} (KKT)} & \phantom{abc} &\multicolumn{3}{c}{Projected Gradient}\\ 
			\cmidrule{2-4} \cmidrule{6-8} \cmidrule{10-12}  \multicolumn{1}{c}{$(n,k)$} & Iterations & Time & Gap && Iterations & Time & Gap && Iterations & Time & Gap\\
	(500,10) & 35 & 1.35 & 1.37e-6 && 35 & 1.72 & 1.37e-6 && 59 & 5.40 & 1.72e-6\\ 
	(1000,20) & 36 &  6.39 & 3.54e-7 && 33 & 8.74 & 1.78e-6 && 50 & 23.21 & 1.76e-6\\
	(2000,50) & 42 & 43.26 & 3.08e-6  && 44 & 65.88 & 1.64e-6 && 74 & 180.57 & 1.83e-6\\
	(4000,50) & 62 &  325.59 & 6.16e-6 && 66 & 511.20 & 8.85e-7 && 161 & 1764.43 & 1.85e-6\\
  (6000,50) & 104 &  1602.37 & 3.18e-5 && 131 & 3001.14 & 1.58e-6 &&  & OOT & \\
%  (6000,50) & 104 &  1602.37 & 3.18e-5 && 131 & 3001.14 & 1.58e-6 && 244 & 7218.99 & 4.70e-5\\
	\bottomrule 
	\end{tabular}}
	\end{table}

It was shown in \cite{yang2013proximal} that the different norm constraint of \eqref{primalexp2-1} (considering $\max$ function as the $\ell_\infty$-norm) can be better in some cases, thus we further investigate the following synthetic problem:
\begin{align}
	 \min_{\X \in \SMAT^n} & \ \ \f(\X) := \C \bullet \X - \log \det \X + 
   % \sum_{h_1=1}^{H} \sum_{h_2=h_1}^{H}
   \sum_{1 \le h_1 \le h_2 \le H}
   \lambda^{'}_{h_1 h_2} \left(\sum_{(i,j) \in G_{h_1h_2}} X_{ij}^2\right)^{\frac{1}{2}},\label{primalexp2-3}
\end{align}
when $\lambda^{'}_{h_1 h_2} := \rho |G_{h_1 h_2}|^{\frac{1}{2}}$. This problem is a variant of \eqref{primalexp2-1} which changes the $max$ function to the Frobenious norm.

Table \ref{table:exp2-sub} shows the performance comparison between the two algorithms on \eqref{primalexp2-3}. % where all the quantities have the same meaning as the table \ref{table:exp2}. 
Focusing on the computation time, Algorithm \ref{alg} again performs well in this problem. In the experiment with $(n,k) = (4000,50)$, Algorithm \ref{alg} executed in 163 seconds and reached a relative gap of $4.98 \times 10^{-15}$, while PG executed in 1038 seconds and gives the solution with a relative gap $2.65 \times 10^{-7}$. We can see that Algorithm \ref{alg} takes a shorter time and outputs a highly accurate solution than PG. % Moreover, the computation for the stopping criterion in Step 1 of Algorithm \ref{alg} is simpler than the one of Algorithm \ref{alg} (KKT) because it needs to compute both primal and dual values in every iteration. This is why Algorithm 1 (KKT) is slower than Algorithm 1.

\begin{table}[H]\centering \ra{1.3}
	\caption{Numerical results on \eqref{primalexp3-1}}
	\label{table:exp2-sub}  
	\addtolength{\tabcolsep}{-1pt}
\scalebox{0.8}{
	\begin{tabular}{@{}rrrrrrrrrrrr@{}}\toprule &\multicolumn{3}{c}{Algorithm~\ref{alg}} & \phantom{abc} &\multicolumn{3}{c}{Algorithm~\ref{alg} (KKT)} & \phantom{abc} &\multicolumn{3}{c}{Projected Gradient}\\ 
			\cmidrule{2-4} \cmidrule{6-8} \cmidrule{10-12}  \multicolumn{1}{c}{$(n,k)$} & Iterations & Time & Gap && Iterations & Time & Gap && Iterations & Time & Gap\\
	(500,10) & 15 & 0.78 & 2.48e-15 && 6 & 0.73 & 8.73e-7 && 6 & 2.01 & 7.65e-7\\ 
	(1000,20) & 16 &  3.11 & 3.07e-14 && 7 & 4.05 & 1.20e-6 && 8 & 11.24 & 5.01e-7\\
	(2000,50) & 22 & 20.20 & 3.35e-14  && 13 & 37.80 & 1.21e-6 && 22 & 108.16 & 7.86e-7\\
	(4000,50) & 31 &  163.00 & 4.98e-15 && 17 & 336.43 & 6.05e-7 && 27 & 1037.94 & 2.65e-7\\
	\bottomrule 
	\end{tabular}}
	\end{table}

 \subsection{Multi-task structure learning problem}\label{subsec:Multitask}

 In this experiment, we present the numerical results on the following multi-task structure learning problem \cite[Equation (3)]{honorio2010multi}:
\begin{align}
	 \min_{\X^1 , \dots , \X^K \in \SMAT^n} & \ \ \sum_{k=1}^{K} \left(\C^k \bullet \X^k - \log \det \X^k\right) + \lambda \sum_{i,j=1}^{n}||(X_{ij}^1 , \dots , \X_{ij}^k)||_{\infty},\label{primalexp3-1}
\end{align}
where $\C^1, \dots ,\C^K \in \SMAT^n$. Following the transformation in  \cite[Section 1.2]{yang2013proximal}, we let $\C := \text{diag}(\C_1,\dots,\C_K) \in \SMAT^{nK}$ and $\X = \text{diag}(\X_1,\dots,\X_K) \in \SMAT^{nK}$. By adding the linear constraint to the non-block diagonal elements $\X$, we can modify \eqref{primalexp3-1} into
\begin{align}
	 \min_{\X^1 , \dots , \X^K \in \SMAT^n} & \ \ \C \bullet \X - \log \det \X + \lambda \sum_{i,j=1}^{n}||(X_{ij}^1 , \dots , X_{ij}^K)||_{\infty} \\
  \mbox{s.t.}  & \ \  X_{ij} = 0 \ \forall (i,j) \in \Omega,  \X \succ \O,\label{primalexp3-2}
\end{align}
where $\Omega := \{(i,j) \ | \ 1 \leq i,j \leq nK , |\ceil*{i/n} - \ceil*{j/n}| \geq 1\}$ with $\ceil*{x}$ being the ceiling function that takes the largest integer which does not exceed $x$. 
We set the penalty $\lambda = 0.005$.
% We can see that \eqref{primalexp3-2} is in the form of \eqref{primal} and can solve it by Algorithm \ref{alg}. % We constructed the covariance matrix $\C^k \in \SMAT^n$ similarly to the experiment of Section~\ref{subsec:fbnormloglikelihood} and 

Table \ref{table:exp3-1} shows the performance of Algorithm \ref{alg} on \eqref{primalexp3-2} with $K =5, 10$, and $15$. We can observe that  Algorithm~\ref{alg} generates accurate solutions even if we increase the value of $K$. This shows the ability of the Algorithm \ref{alg} to solve the problem with a complicated structure by adapting the objective function.

\begin{table}[H] \centering \ra{1.3}
\caption{Numerical results on Multi-task Structure Learning Problem}
	\label{table:exp3-1} 
	\addtolength{\tabcolsep}{-2pt}
\scalebox{0.8}{
	\begin{tabular}{@{}rrrrrrrrrrrr@{}}\toprule &\multicolumn{3}{c}{$K = 5$} & \phantom{abc} &\multicolumn{3}{c}{$K = 10$} & \phantom{abc} &\multicolumn{3}{c}{$K = 15$}\\ 
			\cmidrule{2-4} \cmidrule{6-8} \cmidrule{10-12}  \multicolumn{1}{c}{$n$} & Iterations & Time & Gap && Iterations & Time & Gap && Iterations & Time & Gap\\
	$100$ & 39 & 7.27 & 2.44e-8 && 39 & 14.73 & 9.79e-9 && 32 & 29.57 & 1.18e-8\\ 
	$200$ & 42 & 34.57 & 1.75e-8 && 35 & 79.87 & 5.47e-9&& 32 & 213.43 & 1.34e-8\\
	$300$ & 56 & 114.12 & 4.93e-6 && 36 & 253.93 & 2.90e-8 && 34 & 672.73 & 2.99e-8\\
	$400$ & 55 & 221.69 & 1.64e-5 && 40 & 609.38 & 6.29e-9 && 34 & 1552.03 & 1.68e-8\\
	\bottomrule 
	\end{tabular}}
	\end{table}

 \section{Conclusion}\label{sec:conclusion}

 In this paper, we addressed the generalized log-det SDP \eqref{primal} that covers many existing optimization models and proposed  Algorithm~\ref{alg} based on DSPG. We show the convergence of Algorithm~\ref{alg} to the optimal value under the mild assumptions (Assumption~\ref{assp}). We also provide the results of numerical experiments on the synthetic problem, the number of components in the extension structure (block-constraint), and their combination (multi-task structure). % In all experiments, 
 Algorithm \ref{alg} can obtain accurate solutions on the large instances within the acceptable computation time.

%  \todo{We should revise the following paragraph later again, since its English is not so sophisticated. Charles-understood.}
%  For our future work, we can consider the flexibility of the model \eqref{primal} to establish the further usage of Algorithm~\ref{alg}. \todo{The next sentence does not match the numerical result.} We mention that the projection is the most important factor in deciding whether to use Algorithm~\ref{alg} or not. Therefore, finding a way to compute the inner projection \todo{Algorithm~\ref{alg} does not use alternate projections.(fixed)} in the other models is the way to extend the usage of Algorithm~\ref{alg}, as we did in \cite[Section 4]{namchaisiri2024new} for the log-det SDP with hidden clustering structure. For example, one of such directions is finding the projection to the constraint set created by the extension structure of sparse and locally constant Gaussian graphical models in
% Honorio et al.~\cite{honorio2009sparse}.

One of the future directions is to incorporate squared terms of the $\ell_p$ norm in the objective function like $\|\QC_h(\X)\|_{p_h}^2$. The combination of the $\ell_1$ norm and the squared $\ell_2$ norm appears in statistics. Another important factor is to discuss the types of projection. In the numerical experiments, the main bottleneck was the Cholesky factorization. We still have some flexibility in choosing the projection if the cost in each iteration is at most $O(n^3)$.

\section*{Data Availability}
The test instances in Section 4 were generated randomly following the steps described in these sections.

\section*{Conflict of Interest}
All authors have no conflicts of interest.

\section*{Acknowledgments}
The research of M.~Y.~was partially supported by JSPS KAKENHI (Grant Number: 21K11767).

\begin{small}
\bibliography{ref.bib}
\end{small}
\end{document}